\documentclass[11pt, reqno]{amsart}
\usepackage{amssymb}
\usepackage{verbatim}
\usepackage[vcentermath]{youngtab}
\usepackage{float}
\usepackage{tikz-cd}
\usepackage{mathtools}
\usepackage[margin =1.32in]{geometry}

\usepackage{times}
\usepackage[T1]{fontenc}
\usepackage{mathrsfs}
\usepackage{epsfig}
\usepackage{color}

\newtheorem{thm}{Theorem}
\newtheorem{lemma}[thm]{Lemma}

\newtheorem{prop}[thm]{Proposition}

\theoremstyle{remark}

\theoremstyle{definition}

\makeatletter
\newcommand{\subalign}[1]{%
  \vcenter{%
    \Let@ \restore@math@cr \default@tag
    \baselineskip\fontdimen10 \scriptfont\tw@
    \advance\baselineskip\fontdimen12 \scriptfont\tw@
    \lineskip\thr@@\fontdimen8 \scriptfont\thr@@
    \lineskiplimit\lineskip
    \ialign{\hfil$\m@th\scriptstyle##$&$\m@th\scriptstyle{}##$\crcr
      #1\crcr
    }%
  }
}
\makeatother
\def\ZZ{\mathbb{Z}}
\def\QQ{\mathbb{Q}}
\def\FF{\mathbb{F}}

\def\QQ{\mathbb{Q}}

\def\OO{\mathcal{O}}
\def\p{\mathfrak{p}}

\def\z{\zeta}
\def\gal{\mathrm{Gal}}
\def\aff{\mathrm{Aff}}
\newcommand{\m}{\scalebox{0.6}[1.0]{\( - \)}}

\begin{document}

\title{Collusions in Teichm\"{u}ller expansions}

\author{Trevor Hyde}
\address{Dept. of Mathematics\\
University of Michigan \\
Ann Arbor, MI 48109-1043\\
}
\email{tghyde@umich.edu}

\date{April 25th, 2017}

\maketitle

\begin{abstract}
If $\p \subseteq \ZZ[\z]$ is a prime ideal over $p$ in the $(p^d - 1)$th cyclotomic extension of $\ZZ$, then every element $\alpha$ of the completion $\ZZ[\z]_\p$ has a unique expansion as a power series in $p$ with coefficients in $\mu_{p^d -1} \cup \{0\}$ called the \emph{Teichm\"{u}ller expansion} of $\alpha$ at $\p$. We observe three peculiar and seemingly unrelated patterns that frequently appear in the computation of Teichm\"{u}ller expansions, then develop a unifying theory to explain these patterns in terms of the dynamics of an affine group action on $\ZZ[\z]$.
\end{abstract}

\section{Introduction}

Let $p$ be a prime, $q = p^d$ a power of $p$, and let $\z$ be a primitive $(q-1)$th root of unity. For any prime ideal $\p \subseteq \ZZ[\z]$ over $p$, we have an isomorphism $\ZZ[\z]/\p \cong \FF_q$. Let $\rho: \ZZ[\z] \rightarrow \FF_q$ be the reduction modulo $\p$ map. We call a section $\widetilde{\tau}: \FF_q \rightarrow \ZZ[\z]$ of $\rho$ a \emph{lift}. Given a lift $\widetilde{\tau}$, there is a unique way to expand any element of the completion $\ZZ[\z]_\p$ as a power series in $p$ with coefficients in $\widetilde{\tau}(\FF_q)$.

While there are many lifts of $\rho$, there is a unique \emph{multiplicative} lift $\tau$ called the \emph{Teichm\"{u}ller lift} \cite[Chp. XII, Exer. 16]{lang_alg}. Let $\mu_{q-1}$ denote the multiplicative group of $(q-1)$th roots of unity. Since $q - 1$ is coprime to $p$, the reduction map $\rho$ restricts to an injective group homomorphism $\rho: \mu_{q-1} \rightarrow \FF_q^\times$. Extending the restriction to include 0 we see that $\rho : \mu_{q-1} \cup\{0\} \rightarrow \FF_q$ is an isomorphism of multiplicative monoids. The Teichm\"{u}ller lift $\tau$ is defined as the inverse of this isomorphism, hence is multiplicative. Because $\tau$ is the unique multiplicative lift, $\tau(\FF_q)$ can be seen as a canonical set of coefficients with which to expand elements of $\ZZ[\z]_\p$. We refer to the expansion of an element $\alpha \in \ZZ[\z]_\p$ with respect to $\tau$ as the \emph{Teichm\"{u}ller expansion} of $\alpha$.

The extension of complete local rings $\ZZ[\z]/\ZZ_p$ has degree $d$ where $q = p^d$ and is unramified. There is a unique unramified extension of $\ZZ_p$ of each degree $d$ up to isomorphism \cite[Thm. 3]{Serre}, hence $\ZZ[\z]_\p$ provides a model. Teichm\"{u}ller expansions give a formal way to construct this unique extension of $\ZZ_p$ in terms of $\FF_q$ alone; these are the $p$-typical Witt vectors $W_p(\FF_q)$ \cite[Chp. 2 \S 6]{Serre}. More generally, Teichm\"{u}ller lifts are essential to the construction of the big Witt vectors $W(A)$ for any commutative ring $A$ \cite{witt}. Thus one reason to be interested in Teichm\"{u}ller expansions is to understand the ring structure of Witt vectors $W_p(\FF_q) \cong \ZZ[\z]_\p$. The elements of $\ZZ[\z]_\p$ may be written as power series in $p$ with coefficients in $\tau(\FF_q)$. Our coefficients are closed under multiplication--this is the characteristic property of the Teichm\"{u}ller lift $\tau$--but are not closed under addition. The additive structure is complicated by ``carrying.'' Hence we need to compute the Teichm\"{u}ller expansions of $\z^a + \z^b$ in order to do arithmetic in $\ZZ[\z]_\p$. 

Teichm\"{u}ller expansions are laborious to compute by hand and in the case $q = p$ are less convenient than the usual expression of elements as power series in $p$ with integral coefficients $0\leq i < p$. The difficulty is circumvented with the help of a machine. Peculiar patterns frequently arise as one computes Teichm\"{u}ller expansions. In this paper we observe three seemingly unrelated phenomena and develop a unifying theory to explain them in terms of the dynamics of an affine group action on the global ring $\ZZ[\z]$.

We collect our observations in Section \ref{observe}, followed by a review of cyclotomy in Section \ref{background}. We obtain results in Section \ref{theory} and apply them to explain our examples in Section \ref{examples}.

\section{Observations}
\label{observe}
Recall that $p$ is a prime, $q = p^d$ is a power of $p$, and  $\z$ is a primitive $(q -1)$th root of unity. Given an element $\alpha \in \ZZ[\z]$ and a prime ideal $\p \subseteq \ZZ[\z]$ over $p$, the \emph{Teichm\"{u}ller expansion of $\alpha$ at $\p$} is the unique series
\[
    \alpha = \tau(0,\alpha,\p) + \tau(1,\alpha,\p)p + \tau(2,\alpha,\p)p^2 + \ldots = \sum_{m\geq 0}\tau(m,\alpha,\p)p^m \in \ZZ[\z]_\p,
\]
such that $\tau(m,\alpha, \p) \in \tau(\FF_q) = \mu_{q-1}\cup\{0\}$ for each $m\geq 0$. The $\tau(m,\alpha,\p)$ are called \emph{Teichm\"{u}ller coefficients of $\alpha$ at $\p$}.

To compute explicit Teichm\"{u}ller expansions we must first choose a prime $\p$ over $p$ in $\ZZ[\z]$. The Kummer-Dedekind theorem \cite[Chp 1, 8.3]{neu} says that
\[ 
    \p = (p, f(\zeta)) \subseteq \ZZ[\z]
\]
where $f(x) \in \ZZ[x]$ is congruent modulo $p$ to an irreducible factor of the cyclotomic polynomial $\Phi_{q-1}(x)$ in $\ZZ_p[x]$. There are $\varphi(q-1)/d$ such irreducible factors of degree $d$.

For the moment, let $q = p^d = 2^4$ and let $\z$ be a $q - 1 = 15$th root of unity. The cyclotomic polynomial
\[
    \Phi_{15}(x) = x^8 - x^7 + x^5 - x^4 + x^3 - x + 1
\]
factors into a product of $\varphi(q-1)/d = \varphi(15)/4 = 2$ degree 4 irreducible polynomials over $\ZZ_2$ whose reductions modulo 2 are:
\begin{align*}
    f_1(x) &\equiv x^4 + \,x \,\,+ 1\bmod 2\\
    f_2(x) &\equiv x^4 + x^3 + 1\bmod 2.
\end{align*}
Let $\p_i := (2, f_i(\z))$.

Below are examples of Teichm\"{u}ller expansions at $\p_1$ of sums of two distinct roots of unity.
\begin{align*}
    \z^0 + \z^1 \,\,&= \z^4 + \z^8 p + \,\,\z^6 p^2 + \,\,\,\z^5p^3 + \,\,\z^3 p^4 + 0\, p^5 + \z^8 p^6 + \z^{10} p^7 + \,\,\z^7 p^8 + \z^{10} p^9 + \ldots\\
    \z^1 + \z^3 \,\,&=\z^9 + \z^2 p + \z^{13} p^2 + \z^{11} p^3 + \,\,\z^7 p^4 + 0\, p^5 + \z^2 p^6 + \,\,\z^{6} p^7 + \,\,\z^0 p^8 + \,\,\z^{6} p^9 + \ldots\\
    \z^2 + \z^{10} &= \z^4 + \z^6 p + \,\,\z^5 p^2 + \z^{12} p^3 + \z^{11} p^4 + 0\, p^5 + \z^6 p^6 + \,\,\z^{7} p^7 + \z^{13} p^8 + \,\,\z^{7} p^9 + \ldots\\
    \z^3 + \z^7 \,\,&= \z^{4} + \z^5 p + \z^{12} p^2 + \,\,\z^{8} p^3 + \,\,\z^0 p^4 + 0\, p^5 + \z^5 p^6 + \z^{13} p^7 + \,\,\z^1 p^8 + \z^{13} p^9 + \ldots
\end{align*}
No apparent patterns emerge in the sequence of Teichm\"{u}ller coefficients for an individual expansion. However, we do see some striking relationships between the expansions of different elements. First notice the conspicuous $0$ appearing as the coefficient of $p^5$ in each expansion. Continuing the expansions this phenomenon persists:
\begin{align*}
    \z^0 + \z^1\,\, &= \ldots \,\,\z^2 p^{10} + \z^{11} p^{11} + \mathbf{0}\, p^{12} + \,\,\z^1 p^{13} + \z^{12} p^{14} + \,\,\z^7 p^{15} + \mathbf{0}\, p^{16} + \z^{14} p^{17} + \,\,\z^2 p^{18} + \ldots\\
    \z^1 + \z^3\,\, &= \ldots \,\,\z^5 p^{10} + \,\,\z^8 p^{11}+ \mathbf{0}\, p^{12} + \,\,\z^3 p^{13} + \z^{10} p^{14} + \,\,\z^0 p^{15} + \mathbf{0}\, p^{16} + \z^{14} p^{17} + \,\,\z^5 p^{18} + \ldots\\
    \z^2 + \z^{10} &= \ldots \,\,\z^3 p^{10} + \,\,\z^0 p^{11} + \mathbf{0}\, p^{12} + \z^{10} p^{13} + \,\,\z^8 p^{14} + \z^{13} p^{15} + \mathbf{0}\, p^{16} + \,\,\z^9 p^{17} + \,\,\z^3 p^{18} + \ldots\\
    \z^3 + \z^7\,\, &= \ldots \z^{11} p^{10} + \,\,\z^{2} p^{11} + \mathbf{0}\, p^{12} + \,\,\z^7 p^{13} + \,\,\z^{6} p^{14} +\,\,\z^1 p^{15} + \mathbf{0}\, p^{16} + \z^{14} p^{17} + \z^{11} p^{18}+ \ldots
\end{align*}
Looking closer we see the coefficients of $p^7$ and $p^9$ match in each expansion
\begin{align*}
    \z^0 + \z^1\,\, &= \ldots \z^{10} p^7 + \,\,\z^7 p^8 + \z^{10} p^9 + \ldots\\
    \z^1+ \z^3 \,\,&= \ldots \,\,\z^{6} p^7 + \,\,\z^0 p^8 + \,\,\z^{6} p^9 + \ldots\\
    \z^2 + \z^{10} &= \ldots \,\,\z^{7} p^7 + \z^{13} p^8 + \,\,\z^{7} p^9 + \ldots\\
    \z^3 + \z^7 \,\,&= \ldots \z^{13} p^7 + \,\,\z^1 p^8 + \z^{13} p^9 + \ldots
\end{align*}
 suggesting these expansions may actually be the same under a permutation of the coefficients. The table below supports this claim, showing distributions of the 16 digits in each of the four expansions up to 500 terms.

\begin{center}
\begin{tabular}{|c|r|r|r|r|r|r|r|r|r|r|r|r|r|r|r|r|}
\hline
 \rule{0pt}{1.1em}& $0$ & $\z^0$ & $\z^1$ & $\z^2$ & $\z^3$ & $\z^4$ & $\z^5$ & $\z^6$ & $\z^7$ & $\z^8$ & $\z^9$ & $\z^{10}$ & $\z^{11}$ & $\z^{12}$ & $\z^{13}$ & $\z^{14}$\\ \hline
 \rule{0pt}{1.1em}$\z^0 + \z^1$\,\, & 5.4 & 7.6 & 7.0 & 6.2 & 5.6 & 5.6 & 7.4 & 5.2 & 5.0 & 5.4 & 5.8 & 8.4 & 7.8 & 4.4 & 6.8 & 6.4\\
 $\z^1 + \z^3$\,\, & 5.4 & 5.0 & 7.6 & 5.4 & 7.0 & 5.8 & 6.2 & 8.4 & 5.6 & 7.8 & 5.6 & 4.4 & 7.4 & 6.8 & 5.2 & 6.4\\
 $\z^2 + \z^{10}$ & 5.4 & 7.8 & 6.8 & 7.6 & 6.2 & 5.6 & 5.2 & 5.4 & 8.4 & 4.4 & 6.4 & 7.0 & 5.6 & 7.4& 5.0 & 5.8\\
 $\z^3 + \z^7$\,\, & 5.4 & 5.6 & 5.0 & 7.8 & 7.6 & 5.6 & 5.4 & 4.4 & 7.0 & 7.4 & 5.8 & 6.8 & 6.2 & 5.2 & 8.4 & 6.4\\
 \hline
\end{tabular}
\end{center}

\noindent The rows are in fact permutations of one another with enough distinct entries to almost determine the bijection between them. Notice that the permutations appear to fix zero. We call this phenomenon the \textbf{permutation conspiracy}: seemingly unrelated elements of $\ZZ[\z]$ having the same Teichm\"{u}ller expansion up to a permutation of the coefficients fixing zero. We explain the permutation conspiracy in Section \ref{examples}.

Not every Teichm\"{u}ller expansion of $\z^a + \z^b$ at $\p_1$ is a permutation of one seen above. Here are examples of periodic expansions:
\begin{align*}
    \z^1 + \z^6 \,\,&= \z^{11} + \z^{11} p + \z^{11} p^2 + \z^{11} p^3 + \z^{11} p^4 + \z^{11} p^5 + \z^{11} p^6 + \z^{11} p^7 + \z^{11} p^8 + \z^{11} p^9 + \ldots\\
    \z^4 + \z^{14} &= \,\,\z^9 + \,\,\,\z^9 p + \,\,\z^9 p^2 + \,\,\z^9 p^3 + \,\,\,\z^9 p^4 + \,\,\z^9 p^5 + \,\,\z^9 p^6 + \,\,\,\z^9 p^7 + \,\,\z^9 p^8 + \,\,\z^9 p^9 + \ldots
\end{align*}

\noindent Note that the exponents on the left hand side differ by a multiple of 5 in each case. 

The following expansions are related by a permutation conspiracy but also have \textbf{restricted coefficients} taken from the set $\{0, \z^4, \z^9, \z^{14}\}$.
\begin{align*}
    \z^0 + \z^{3} \,\,&= \z^{14} + \,\,\z^9 p + \z^4 p^2 + \,\,\z^{9} p^3 + \z^{14} p^4 + 0\, p^5 + 0\, p^6 + \,\,\z^{9} p^7 + \z^{14} p^8 + 0\, p^9 + \ldots\\
    \z^2 + \z^{11} &= \,\,\z^{9} + \z^{14} p + \z^4 p^2 + \z^{14} p^3 + \,\,\z^{9} p^4 + 0\, p^5 + 0\, p^6 + \z^{14} p^7 + \,\,\z^{9} p^8 + 0\, p^9 + \ldots\\
    \z^1 + \z^{7} \,\,&= \z^{14} + \,\,\z^{4} p + \z^9 p^2 + \,\,\z^{4} p^3 + \z^{14} p^4 + 0\, p^5 + 0\, p^6 + \,\,\z^{4} p^7 + \z^{14} p^8 + 0\, p^9 + \ldots
\end{align*}

\noindent The exponents on the left hand side differ by multiples of 3 in each case.

How do we account for these special expansions with periodic or restricted coefficients? Can we predict when such phenomena will occur and what coefficients will appear? An affirmative answer is provided in Section \ref{examples}.

Still working with $q = 16$, we now compare the Teichm\"{u}ller expansions of an element at both of the primes $\p_1, \p_2$ over $p$.
\begin{align*}
    \p_1: \z^0 + \z^1 \,\,&= \,\,\z^4 + \,\,\z^8 p + \,\,\,\z^6 p^2 + \,\,\z^5 p^3 + \,\,\z^3 p^4 + 0\, p^5 + \z^8 p^6 + \z^{10} p^7 + \z^7 p^8 + \z^{10} p^9 + \ldots\\
    \p_2: \z^0 + \z^1 \,\,&= \z^{12} + \,\,\z^8 p + \z^{10} p^2 + \z^{11} p^3 + \z^{13} p^4 + 0\, p^5 + \z^8 p^6 + \,\,\z^{6} p^7 + \z^9 p^8 + \,\,\z^{6} p^9 + \ldots\\
    &\\
    \p_1: \z^4 + \z^{14} &=\,\,\z^9 + \,\,\z^9 p + \,\,\z^9 p^2 + \,\,\z^9 p^3 + \,\,\z^9 p^4 + \z^9 p^5 + \z^9 p^6 + \,\,\z^9 p^7 + \z^9 p^8 + \,\,\z^9 p^9 + \ldots\\
    \p_2: \z^4 + \z^{14} &=\,\,\z^9 + \,\,\z^9 p + \,\,\z^9 p^2 + \,\,\z^9 p^3 + \,\,\z^9 p^4 + \z^9 p^5 + \z^9 p^6 + \,\,\z^9 p^7 + \z^9 p^8 + \,\,\z^9 p^9 + \ldots\\
    &\\
    \p_1: \z^2 + \z^{11} &= \,\,\z^{9} + \z^{14} p + \,\,\z^4 p^2 + \z^{14} p^3 + \,\,\z^{9} p^4 + 0\, p^5 + \,\,0\, p^6 + \z^{14} p^7 + \z^{9} p^8 + \,\,\,\,0\, p^9 + \ldots\\
    \p_2: \z^2 + \z^{11} &= \,\,\z^{4} + \z^{14} p + \,\,\z^9 p^2 + \z^{14} p^3 + \,\,\z^{4} p^4 + 0\, p^5 + \,\,0\, p^6 + \z^{14} p^7 + \z^{4} p^8 + \,\,\,\,0\, p^9 + \ldots
\end{align*}
In each example, the product $\tau(m,\alpha,\p_1)\tau(m,\alpha,\p_2)$ is independent of $m$ whenever its nonzero. The values of the products are $\z^1, \z^3, \z^{13}$ respectively (recall that $\z$ is a 15th root of unity.) We refer to this relationship between the Teichm\"{u}ller coefficients of $\alpha$ at different primes as \textbf{prime collusion}. 

To get a better sense of prime collusion, let us consider examples when $q = p^d = 2^6$. Then $\zeta$ is a 63rd root of unity. The polynomial $\Phi_{63}(x)$ factors into $\varphi(63)/6 = 6$ degree 6 irreducible polynomials in $\ZZ_2[x]$.
\begin{center}
\begin{tabular}{rcrcrcr}
     $g_1(x)$&$\equiv$& $x^6 + x^5 + x^4 + x + 1 \bmod 2 $& \hspace{2em} & $g_2(x)$&$\equiv$&$ x^6 + x + 1 \bmod 2$\\
     $g_3(x)$&$\equiv$&$ x^6 + x^5 + x^3 + x^2 + 1 \bmod 2 $& \hspace{2em} & $g_4(x)$&$\equiv$&$ x^6 + x^4 + x^3 + x + 1 \bmod 2$\\
     $g_5(x)$&$\equiv$&$ x^6 + x^5 + 1 \bmod 2$ & \hspace{2em} & $g_6(x)$&$\equiv$&$ x^6 + x^5 + x^2 + x + 1 \bmod 2$
\end{tabular}
\end{center}
Let $\p_i = \big(2, g_i(\zeta)\big)$. Each element has 6 expansions, for instance:
\begin{align*}
    \p_1: \z^0 + \z^1 &=\z^{39} + \z^{32} p + \,\,\z^{4} p^2 + \z^{53} p^3 + \z^{35} p^4 + \,\,\z^{2} p^5 + \,\,\z^{2} p^6 + \z^{44} p^7 + \z^{39} p^8 + \,\,\z^{2} p^9 + \ldots\\
    \p_6: \z^0 + \z^1 &=\z^{25} + \z^{32} p + \z^{60} p^2 + \z^{11} p^3 + \z^{29} p^4 + \z^{62} p^5 + \z^{62} p^6 + \z^{20} p^7 + \z^{25} p^8 + \z^{62} p^9 + \ldots\\
    &\\
    \p_2: \z^0 + \z^1 &=\,\,\z^{6} + \z^{32} p + \z^{19} p^2 + \z^{44} p^3 + \,\,\z^{1} p^4 + \z^{11} p^5 + \z^{24} p^6 + \z^{29} p^7 + \z^{20} p^8 + \z^{16} p^9 + \ldots\\
    \p_5: \z^0 + \z^1 &=\z^{58} + \z^{32} p + \z^{45} p^2 + \z^{20} p^3 + \,\,\z^{0} p^4 + \z^{53} p^5 + \z^{40} p^6 + \z^{35} p^7 + \z^{44} p^8 + \z^{48} p^9 + \ldots\\
    &\\
    \p_3: \z^0 + \z^1 &=\,\,\z^{8} + \z^{32} p + \z^{20} p^2 + \z^{14} p^3 + \z^{51} p^4 + \z^{20} p^5 + \z^{40} p^6 + \z^{37} p^7 + \z^{40} p^8 + \z^{45} p^9 + \ldots\\
    \p_4: \z^0 + \z^1 &=\z^{56} + \z^{32} p + \z^{44} p^2 + \z^{50} p^3 + \z^{13} p^4 + \z^{44} p^5 + \z^{24} p^6 + \z^{27} p^7 + \z^{24} p^8 + \z^{19} p^9 + \ldots
\end{align*}
Again we have a constant product:
\[
    \prod_{1\leq i\leq 6} \tau(m,\z^0+\z^1,\p_i) = \z^3
\]
for all $m\geq 0$ when the product is nonzero. However, observe that for $i = 1, 2, 3$,
\[
    \tau(m,\z^0+\z^1,\p_i)\tau(m,\z^0+\z^1,\p_{7-i}) = \z^1.
\]
These pairwise products refine the collusion noted between all 6 primes. This may lead us to expect collusions to come in pairs, but the next example shows collusion in triples of primes:
\begin{align*}
    \p_1: \z^4 + \z^{37} + \z^{43}  &=\z^{56} + \z^{35} p + \z^{56} p^2 + \z^{14} p^3 + \z^{14} p^4 + \z^{35} p^5 + \z^{35} p^6 + \z^{56} p^7 + \z^{14} p^8 + \ldots\\
    \p_2: \z^4 + \z^{37} + \z^{43}  &=\z^{35} + \z^{14} p + \z^{35} p^2 + \z^{56} p^3 + \z^{56} p^4 + \z^{14} p^5 + \z^{14} p^6 + \z^{35} p^7 + \z^{56} p^8 + \ldots\\
    \p_3: \z^4 + \z^{37} + \z^{43}  &=\z^{14} + \z^{56} p + \z^{14} p^2 + \z^{35} p^3 + \z^{35} p^4 + \z^{56} p^5 + \z^{56} p^6 + \z^{14} p^7 + \z^{35} p^8 + \ldots\\
    &\\
    \p_4: \z^4 + \z^{37} + \z^{43}  &=\z^{35} + \z^{35} p + \z^{56} p^2 + \z^{14} p^3 + \z^{35} p^4 + \,\,\,0\, p^5\, + \,\,\,0\, p^6 + \z^{35} p^7 + \z^{14} p^8 + \ldots\\
    \p_5: \z^4 + \z^{37} + \z^{43}  &=\z^{56} + \z^{56} p + \z^{14} p^2 + \z^{35} p^3 + \z^{56} p^4 + \,\,\,0\, p^5\, + \,\,\,0\, p^6 + \z^{56} p^7 + \z^{35} p^8 + \ldots\\
    \p_6: \z^4 + \z^{37} + \z^{43}  &=\z^{14} + \z^{14} p + \z^{35} p^2 + \z^{56} p^3 + \z^{14} p^4 + \,\,\,0\, p^5\, + \,\,\,0\, p^6 + \z^{14} p^7 + \z^{56} p^8 + \ldots
\end{align*}
Here we have constant triple products for $\alpha = \z^4 + \z^{37} + \z^{43}$ and $i = 1, 4$:
\[
    \tau(m,\alpha, \p_i)\tau(m,\alpha, \p_{i+1})\tau(m,\alpha, \p_{i+2}) = \z^{42}.
\]
Furthermore, $\z^4 + \z^{37} + \z^{43}$ gives another example of the restricted coefficient phenomenon, since $\{0,\z^{14}, \z^{35}, \z^{56}\}$ are the only coefficients appearing in any of these expansions. Note that in all cases we have seen of the restricted coefficients phenomenon, the total number of permissible Teichm\"{u}ller representatives has been a power of $p$.

The remainder of the paper is divided into three sections. Section \ref{background} reviews background on cyclotomic fields. Section \ref{theory} develops theory which we use in Section \ref{examples} to explain the permutation conspiracy, restricted coefficients phenomenon, and prime collusion. All three are related to the dynamics of an affine group action on $\ZZ[\z]$ established in Theorem \ref{teich}.

\section{Background}
\label{background}
We review the basic theory of cyclotomic fields. Proofs may be found in Lang \cite[Chp. IV]{Lang}. To begin, the polynomial $x^n - 1$ factors in $\ZZ[x]$ as
\[
    x^n - 1 = \prod_{d\mid n} \Phi_d(x),
\]
where $\Phi_d(x)$ is an irreducible polynomial of degree $\varphi(d)$ called the \emph{$d$th cyclotomic polynomial}. Recall that $\varphi(d) = |(\ZZ/(d))^\times|$ is Euler's totient function. The roots of $\Phi_n(x)$ are primitive $n$th roots of unity: algebraic integers $\zeta$ such that $\z^n = 1$ and $\z^d \neq 1$ for any proper divisor $d\mid n$. If $\z$ is a primitive $n$th root of unity and $a$ is an integer coprime to $n$, then $\zeta^a$ is again a primitive $n$th root of unity. Furthermore, every primitive $n$th root of unity may be expressed as $\zeta^a$ for some $a$ with $(a,n) = 1$. Therefore $\Phi_n(x)$ splits completely in $\QQ(\z)$ and $\QQ(\zeta)/\QQ$ is a Galois extension of degree $\varphi(n)$ with Galois group canonically isomorphic to $(\ZZ/(n))^\times$. We use this isomorphism frequently without further comment, writing $\alpha^\sigma$ to denote the action of $\sigma \in \gal(\QQ(\z)/\QQ)$ but also viewing $\sigma$ as a unit modulo $n$ as in $\z^\sigma$.

If $p$ is a prime, then $\Phi_n(x)$ is typically not irreducible in $\ZZ_p[x]$. By Hensel's lemma \cite[Lem. 4.6]{neu}, the factorization is determined by the orbits of Frobenius on the primitive $n$th roots of unity. If $(n,p) = 1$, then $p \in (\ZZ/(n))^\times \cong \gal(\QQ(\zeta)/\QQ)$ and the size of the orbits of Frobenius on the primitive $n$th roots of unity is the same as the multiplicative order of $p$ modulo $n$. That is, if $d>0$ is minimal such that $p^d \equiv 1 \bmod n$, then $\Phi_n(x)$ factors into $\varphi(n)/d$ irreducible degree $d$ polynomials in $\ZZ_p[x]$. The case of interest to us is when $n = q - 1 = p^d - 1$. With $n$ written in this form it is clear that $(n, p) = 1$ and $d$ is the multiplicative order of $p$ modulo $n$.

The \emph{affine group} $\aff(n)$ is defined by
\[
    \aff(n) = \{\sigma x + b : \sigma \in (\ZZ/(n))^\times, b \in \ZZ/(n)\},
\]
where the elements are considered as linear functions in the formal variable $x$ and the group operation is composition of functions:
\[
    (\sigma_1 x + b_1) \circ (\sigma_2 x + b_2) = \sigma_1(\sigma_2 x + b_2) + b_1 = (\sigma_1\sigma_2) x + (\sigma_1 b_2 + b_1).
\]
Another way to view $\aff(n)$ is as a semidirect product
\[
    \aff(n) \cong \ZZ/(n) \rtimes (\ZZ/(n))^\times.
\]
If $H \subseteq (\ZZ/(n))^\times $ is a subgroup, then $\aff(H) \subseteq \aff(n)$ is the subgroup
\[
    \aff(H) = \{\sigma x + b \in \aff(n) : \sigma \in H\}.
\]
There is a $\ZZ$-linear action of $\aff(n)$ on $\ZZ[\zeta]$ given by
\[
    \sigma x + b: \alpha \longmapsto \alpha^\sigma \zeta^b.
\]
In particular, if $\alpha = \zeta^a$, then $(\sigma x + b)\zeta^a = \zeta^{\sigma a + b}$.

If $n = q - 1 = p^d - 1$, let $H_p \subseteq (\ZZ/(q-1))^\times$ be the subgroup generated by $p$. If $\p$ is a prime in $\ZZ[\z]$ over $p$, then $(p^c x + b)\p = \p$ for each $p^c x + b \in \aff(H_p)$ since $H_p$ is the decomposition group of each $\p$ over $p$.

In Section \ref{theory} we are interested in the number of fixed points of an element $p^c x + b \in \aff(H_p)$ in $\mu_{q-1}\cup\{0\}$. We determine this in Proposition \ref{fixed}. First, a lemma.

\begin{lemma}
\label{gcd}
There exist polynomials $f(x), g(x) \in \ZZ[x]$ such that
\begin{equation}
\label{gcd_id}
    f(x)\frac{x^a-1}{x-1} + g(x)\frac{x^b-1}{x-1} = \frac{x^{(a,b)}-1}{x-1}.
\end{equation}
It follows that for any integer $m$ we have
\begin{equation}
\label{gcd_id2}
    (m^a - 1, m^b - 1) = m^{(a,b)}-1.
\end{equation}
\end{lemma}

\begin{proof}
If $a = qb + r$ with $0\leq r <b$, then
\[
    \frac{x^a - 1}{x-1} = \big(x^{a-b} + x^{a-2b} + \ldots + x^{a-qb}\big)\frac{x^b-1}{x-1} + \frac{x^r - 1}{x-1}.
\]
Thus we can follow the usual Euclidean algorithm to get the desired polynomial identity \eqref{gcd_id}. Dividing \eqref{gcd_id} by $(x^{(a,b)}-1)/(x -1)$ we have
\[
    f(x)\frac{x^a-1}{x^{(a,b)}-1} + g(x)\frac{x^b-1}{x^{(a,b)}-1} = 1
\]
in $\ZZ[x]$. Evaluating at $x = m$ we deduce
\[
    \left(\frac{m^a-1}{m^{(a,b)}-1}, \frac{m^b-1}{m^{(a,b)}-1}\right) = 1.
\]
Multiplying by $m^{(a,b)} - 1$ yields the identity \eqref{gcd_id2}. 
\end{proof}

\begin{prop}
\label{fixed}
The element $p^c x + b \in \aff(H_p)$ has fixed points in $\mu_{q-1}$ iff $p^{(c,d)} -1 \mid b$, and in that case it has precisely $p^{(c,d)}-1$ fixed points in $\mu_{q-1}$. Since $0$ is always fixed, it follows that the total number of fixed points in $\tau(\FF_q) = \mu_{q-1}\cup\{0\}$ is $p^{(c,d)}$.
\end{prop}

\begin{proof}
If $\z^a$ is a fixed point of $p^c x + b$, then $p^c a + b = a$ hence $(p^c - 1)a + b = 0$ in $\ZZ/(q-1)$. Let $g = p^{(c,d)} -1$. Then $g = (p^c - 1, p^d -1)$ by Lemma \ref{gcd}. Reducing $(p^c - 1)y + b = 0$ modulo $g$ we conclude that $g \mid b$.

Supposing $g \mid b$, the linear equation $\frac{p^c - 1}{g}x + \frac{b}{g} = 0$ has a unique solution modulo $\frac{q-1}{g}$ which lifts to $g$ distinct solutions modulo $q-1$. Thus $p^c x + b$ has a total of $g + 1 = p^{(c,d)}$ fixed points in $\mu_{q-1}\cup\{0\}$.
\end{proof}

\section{Theory}
\label{theory}

The Teichm\"{u}ller expansion of $\alpha$ at $\p$ is defined locally as an infinite sum which does not converge in the global ring $\ZZ[\zeta]$. Nevertheless, Theorem \ref{teich} shows that the global Galois group acts nicely on Teichm\"{u}ller expansions.

Recall our definition of the affine group $\aff(n)$ and its subgroup $\aff(H_p)$ assuming $(n,p) = 1$,
\begin{align*}
    \aff(n) &= \{\sigma x + b: \sigma \in (\ZZ/(n))^\times, b\in\ZZ/(n)\}\\
    \aff(H_p) &= \{p^c x + b \in \aff(n): c \geq 0\}.
\end{align*}

\begin{thm}
\label{teich}
Let $\zeta$ be a primitive $(q-1)$th root of unity. Suppose $\p \subseteq \ZZ[\zeta]$ is a prime over $p$, and $\alpha \in \ZZ[\z]$. If $\sigma x + b \in \aff(n)$, then 
\[
    (\sigma x + b)\tau(m,\alpha,\p) = \tau(m,(\sigma x + b)\alpha,\p^\sigma).
\]
In other words, if
\[
    \alpha = \tau(0,\alpha,\p) + \tau(1,\alpha,\p)p + \tau(2,\alpha,\p)p^2 + \ldots
\]
is the Teichm\"{u}ller expansion of $\alpha$ at $\p$, then
\[
    (\sigma x + b)\alpha = (\sigma x + b)\tau(0,\alpha,\p) + (\sigma x + b)\tau(1,\alpha,\p)p + (\sigma x + b)\tau(2,\alpha,\p)p^2 + \ldots
\]
is the Teichm\"{u}ller expansion of $(\sigma x + b)\alpha$ at $\p^\sigma$.
\end{thm}

\begin{proof}
Let $\alpha(m,\p)$ be the sum of the first $m$ terms of the Teichm\"{u}ller expansion of $\alpha$ at $\p$. Then
\[
    \alpha(m,\p) = \sum_{k< m}\tau(k,\alpha,\p)p^k
\]
is the unique element of $\ZZ[\zeta]$ which may be written as a polynomial in $p$ of degree less than $m$ with coefficients in $\tau(\FF_q) = \mu_{q-1} \cup\{0\}$ such that $\alpha - \alpha(m,\p) \in \p^m$. If $\sigma \in \gal(\QQ(\zeta)/\QQ)$, then $\alpha^\sigma\z^b - \alpha(m,\p)^\sigma\z^b \in (\p^\sigma)^m$. Since $p \in \ZZ$ is fixed by $\sigma$ we have
\[
    (\sigma x + b)\alpha(m,\p) = \alpha(m,\p)^\sigma \z^b = \sum_{k< m}\tau(k,\alpha,\p)^\sigma \z^b p^k,
\]
which is a polynomial in $p$ of degree less than $m$ with coefficients in $\tau(\FF_q)$. Hence
\[
    (\sigma x + b)\alpha(m,\p) = \alpha^\sigma(m,\p^\sigma)\z^b
\]
by uniqueness. This implies that $(\sigma x + b)\tau(m,\alpha,\p) = \tau(m, (\sigma x + b)\alpha, \p^\sigma)$ for each $m\geq 0$.
\end{proof}

\noindent To summarize Theorem \ref{teich}, the affine group $\aff(n)$ acts coordinatewise on Teichm\"{u}ller expansions while permuting the primes $\p$ over $p$.

Next we deduce three results using Theorem \ref{teich} to explain the permutation conspiracy, restricted coefficient phenomenon, and prime collusion in Section \ref{examples}. Proposition \ref{permute} applies to the permutation conspiracy.

\begin{prop}
\label{permute}
If $\alpha \in \ZZ[\z]$ and $\sigma x + b \in \aff(H_p)$, then
\[
    \tau(m,(\sigma x + b)\alpha, \p) = (\sigma x + b) \tau(m, \alpha, \p).
\]
Hence the Teichm\"{u}ller expansion of $(\sigma x + b)\alpha$ at $\p$ is the same as that of $\alpha$ at $\p$ up to a permutation of the coefficients fixing 0.
\end{prop}

\begin{proof}
By Theorem \ref{teich} we have 
\[
    (\sigma x + b) \tau(m, \alpha, \p) = \tau(m,(\sigma x + b)\alpha, \p^\sigma) = \tau(m, (\sigma x + b)\alpha, \p).
\]
since $\sigma = p^c$ fixes $\p$ (see Section \ref{background}). Note that every element of $\aff(q-1)$ fixes $0$ and the rest of the permutation claim follows from $\tau(\FF_q)$ being closed under the action of $\aff(q-1)$.
\end{proof}

Proposition \ref{restrict} helps us understand the restricted coefficients phenomenon.

\begin{prop}
\label{restrict}
 If $\alpha \in \ZZ[\zeta]$ is invariant under $p^c x + b \in \aff(H_p)$, then the Teichm\"{u}ller coefficients of $\alpha$ at any prime $\p$ over $p$ are fixed points of $p^c x + b$. If $\alpha \neq 0$, then $p^{(c,d)}-1 \mid b$ and there are $p^{(c,d)}$ fixed points of $p^c x + b$ in $\tau(\FF_q)$.
\end{prop}

\begin{proof}
Theorem \ref{teich} implies that for $\sigma = p^c$,
\[
    (\sigma x + b)\tau(m,\alpha, \p) = \tau(m,(\sigma x + b)\alpha, \p^\sigma) = \tau(m,\alpha, \p),
\]
so the Teichm\"{u}ller coefficients of $\alpha$ at $\p$ are fixed points of $p^c x + b$. If $\alpha\neq 0$, then there is some nonzero Teichm\"{u}ller coefficient. Hence $p^c x + b$ has fixed points in $\mu_{q-1}$. Proposition \ref{fixed} tells us $p^{(c,d)} -1 \mid b$ and that the total number of fixed points in $\tau(\FF_q)$ is $p^{(c,d)}$.
\end{proof}

Finally, Proposition \ref{sub} explains prime collusion.
\begin{prop}
\label{sub}
Suppose $\alpha \in \ZZ[\zeta]$ is invariant under an element $\sigma x + b \in \aff(q-1)$ of order $k$. Then for any $m\geq 0$ and prime $\p$ over $p$,
\begin{enumerate}
    \item If $\tau(m,\alpha,\p) = 0$, then $\tau(m, \alpha,\p^{\sigma^i}) = 0$ for all $i\geq 0$.
    \item If $\tau(m,\alpha,\p)\neq 0$, then
    \[
        \prod_{0\leq i < k}\tau(m,\alpha,\p^{\sigma^i}) = (ux + vb)\tau(m,\alpha,\p)
    \]
    where $u, v$ are given by
    \[
        u \equiv  \sum_{0\leq i < k}\sigma^i \bmod q-1 \hspace{3em}v \equiv \sum_{0\leq i < k}\sum_{1\leq j \leq i}\sigma^j \bmod q -1
    \]
    Alternatively, $u,v \in \ZZ/(q-1)$ are the unique elements such that for each maximal prime power divisor $\ell^n \mid q-1$,
    \begin{enumerate}
        \item If $\sigma \equiv 1 \bmod \ell^n$, then
        \[
            u  \equiv  k \bmod \ell^n \hspace{3em}
            v  \equiv  \,\,\frac{k(k+1)}{2} \bmod \ell^n
        \]
        \item If $\sigma \not\equiv 1 \bmod \ell^n$ then
        \[
            u \equiv  \,\,0 \bmod \ell^r \hspace{3em}
            (1 - \sigma)v \equiv  \,\,k\bmod \ell^r
        \]
        where $r = n - v_\ell(1 - \sigma)$ and $v_\ell(a)$ denotes the normalized $\ell$-valuation of $a$.
    \end{enumerate}
\end{enumerate}
\end{prop}

\begin{proof}
Let $(\sigma x + b)^i$ denote the $i$th iterate of $\sigma x+ b$ in $\aff(q-1)$. Then Theorem \ref{teich} implies that for each $m\geq 0$
\[
    (\sigma x + b)^i\tau(m,\alpha, \p) = \tau(m,(\sigma x + b)^i\alpha, \p^{\sigma^i}) = \tau(m,\alpha, \p^{\sigma^i}),
\]
since $\alpha$ is fixed by $\sigma x + b$. So (1) follows from $0$ being fixed by $\aff(q-1)$. 

Now suppose $\tau(m,\alpha,\p)\neq 0$. Then
\begin{align*}
    \prod_{0\leq i < k}\tau(m,\alpha,\p^{\sigma^i}) &= \prod_{0\leq i < k}(\sigma x + b)^i\tau(m,\alpha,\p)\\
    &= \prod_{0\leq i < k}\tau(m,\alpha,\p)^{\sigma^i} \z^{\sum_{1\leq j\leq i}\sigma^j b}\\
    &= \tau(m,\alpha, \p)^{\sum_{0\leq i < k}\sigma^i} \z^{\sum_{0\leq i < k}\sum_{1\leq j \leq i} \sigma^j b}\\
    &= (ux + vb)\tau(m,\alpha, \p),
\end{align*}
where
\[
    u \equiv  \sum_{0\leq i < k}\sigma^i \bmod q-1 \hspace{3em}v \equiv \sum_{0\leq i < k}\sum_{1\leq j \leq i}\sigma^j \bmod q -1.
\]
Note that $u$ may be 0 in which case we interpret the action of $(ux + vb)$ as simply multiplication by $\z^{vb}$. It does not seem possible to find simple evaluations for these sums modulo $q - 1$, but we can do so modulo the maximal prime power divisors $\ell^n \mid q - 1$ and then use the Chinese Remainder Theorem to show that these local computations uniquely determine $u$ and $v$.

If $\sigma \equiv 1 \bmod \ell^n$, then the sums simplify to well-known values,
\[
    u  \equiv  k \bmod \ell^n \hspace{3em}
    v  \equiv  \,\,\frac{k(k+1)}{2} \bmod \ell^n.
\]
If $\sigma \not\equiv 1 \bmod \ell^n$, then
\[
    (1-\sigma)u \equiv 1 - \sigma^k \equiv 0 \bmod \ell^n,
\]
which implies $u \equiv 0 \bmod \ell^r$ with $r = n - v_\ell(1 - \sigma)$.

Next we compute
\[
    v = \sum_{0\leq i < k}\sum_{1\leq j \leq i}\sigma^j = \sum_{0\leq j < k}\sum_{j\leq i < k}\sigma^j
    = \sum_{0\leq j < k}(k-j)\sigma^j
    =ku -\sum_{0\leq j < k}j\sigma^j,
\]
where the first equality results from switching the order of summation and reindexing. Multiplying by $1 - \sigma$ yields
\[
    (1 - \sigma)v = k(1 - \sigma)u - (1 - \sigma)\sum_{0\leq j < k} j\sigma^j
    \equiv k + \sum_{0\leq j < k}\sigma^j \equiv k + u \bmod \ell^n.
\]
Multiplying by $1 - \sigma$ again we have
\[ 
    (1 - \sigma)^2v = (1 - \sigma)k \bmod \ell^n.
\]
So $(1 - \sigma)v \equiv k \bmod \ell^r$.
\end{proof}

\section{Application}
\label{examples}
We revisit the examples from Section \ref{observe}, applying the results from Section \ref{theory} to explain the conspiracies and collusion.

\subsection*{Permutation conspiracy}
Let $q = p^d$ and let $\zeta$ be a $(q-1)$th root of unity. Proposition \ref{permute} implies that the Teichm\"{u}ller expansion at $\p$ of any element in the $\aff(H_p)$ orbit of $\alpha$ is a permutation of the Teichm\"{u}ller expansion of $\alpha$ at $\p$ fixing 0.

Recall these Teichm\"{u}ller expansions at $\p_1 = \big(2, \zeta^4 + \zeta + 1\big)$ when $q = 2^4$ from Section \ref{observe}:
\begin{align*}
    \z^0 + \z^1 \,\,&= \z^4 + \z^8 p + \,\,\z^6 p^2 + \,\,\,\z^5p^3 + \,\,\z^3 p^4 + 0\, p^5 + \z^8 p^6 + \z^{10} p^7 + \,\,\z^7 p^8 + \z^{10} p^9 + \ldots\\
    \z^1 + \z^3 \,\,&=\z^9 + \z^2 p + \z^{13} p^2 + \z^{11} p^3 + \,\,\z^7 p^4 + 0\, p^5 + \z^2 p^6 + \,\,\z^{6} p^7 + \,\,\z^0 p^8 + \,\,\z^{6} p^9 + \ldots\\
    \z^2 + \z^{10} &= \z^4 + \z^6 p + \,\,\z^5 p^2 + \z^{12} p^3 + \z^{11} p^4 + 0\, p^5 + \z^6 p^6 + \,\,\z^{7} p^7 + \z^{13} p^8 + \,\,\z^{7} p^9 + \ldots\\
    \z^3 + \z^7 \,\,&= \z^{4} + \z^5 p + \z^{12} p^2 + \,\,\z^{8} p^3 + \,\,\z^0 p^4 + 0\, p^5 + \z^5 p^6 + \z^{13} p^7 + \,\,\z^1 p^8 + \z^{13} p^9 + \ldots
\end{align*}
The permutation conspiracies are consequences of the following calculations:
\[
    (2x + 1)(\z^0 + \z^1) = \z^1 + \z^3 \hspace{1.5em} (8x + 2)(\z^0 + \z^1) = \z^2 + \z^{10} \hspace{1.5em} (4x + 3)(\z^0 + \z^1) = \z^3 + \z^7
\]
It's important to note that each element of $\aff(15)$ above has the form $2^c x + b$.

The permutations are determined explicitly by the linear functions; applying $2x + 1$ to the exponents of the Teichm\"{u}ller coefficients in $\z^0 + \z^1$ yields the expansion of $\z^1 + \z^3$ below it. The group $\aff(H_2)$ has order $d(q - 1) = 4\cdot 15 = 60$ and the element $\z^0 + \z^1$ has a trivial stabilizer, thus an orbit with 60 elements. Therefore, of the $\binom{q-1}{2} = 105$ sums of the form $\z^a + \z^b$ with $a\not\equiv b \bmod 15$, approximately  $57\%$ of them will be permutations of the expansion of $\z^0 + \z^1$. More generally for any $q$, $\z^0 + \z^1$ always has trivial stabilizer under $\aff(H_p)$, hence the proportion of $\z^a + \z^b$ which are permutations of $\z^0 + \z^1$ is
\[
    \frac{d(q-1)}{\binom{q-1}{2}} = \frac{2d}{q-2}.
\]

We saw two periodic expansions
\begin{align*}
    \z^1 + \z^6 \,\,&= \z^{11} + \z^{11}p + \z^{11}p^2 + \ldots\\
    \z^4 + \z^{14} &= \,\,\z^9 + \,\,\,\z^9 p + \,\,\z^9 p^2 + \ldots
\end{align*}
which are related by $(2x + 2)(\z^1 + \z^6) = \z^4 + \z^{14}$. The periodic expansion itself is special and can be understood by summing the geometric series which converges locally:
\[
\z^1 + \z^6 = \z^{11} + \z^{11}p + \z^{11}p^2 + \ldots = \z^{11}(1 + p + p^2 + \ldots ) = \frac{\z^{11}}{1 -  p} = -\z^{11}.
\]
This identity is equivalent to $\z^{10} + \z^5 + 1 = 0$, telling us $\z^5$ is a primitive 3rd root of unity.

\subsection*{Restricted coefficients}
Our last example of the permutation conspiracy also exhibited the restricted coefficient phenomenon.
\begin{align*}
    \z^0 + \z^{3} \,\,&= \z^{14} + \,\,\z^9 p + \z^4 p^2 + \,\,\z^{9} p^3 + \z^{14} p^4 + 0\, p^5 + 0\, p^6 + \,\,\z^{9} p^7 + \z^{14} p^8 + 0\, p^9 + \ldots\\
    \z^2 + \z^{11} &= \,\,\z^{9} + \z^{14} p + \z^4 p^2 + \z^{14} p^3 + \,\,\z^{9} p^4 + 0\, p^5 + 0\, p^6 + \z^{14} p^7 + \,\,\z^{9} p^8 + 0\, p^9 + \ldots\\
    \z^1 + \z^{7} \,\,&= \z^{14} + \,\,\z^{4} p + \z^9 p^2 + \,\,\z^{4} p^3 + \z^{14} p^4 + 0\, p^5 + 0\, p^6 + \,\,\z^{4} p^7 + \z^{14} p^8 + 0\, p^9 + \ldots
\end{align*}
The permutations follow from 
\begin{align*}
    (8x + 2)(\z^0 + \z^3) &= \z^2 + \z^{11}\\
    (2x + 1)(\z^0 + \z^3) &= \z^1 + \z^7.
\end{align*}
To see why the coefficients all belong to $\{0,\z^4, \z^9, \z^{14}\}$, notice that $\z^0 + \z^3$ is invariant under $4x + 3 \in \aff(H_2)$. Proposition \ref{permute} says the coefficients of the Teichm\"{u}ller expansions of $\z^0 + \z^3$ at both primes $\p_1$ and $\p_2$ are invariant under $4x + 3$. The fixed points of $4x + 3$ in $\tau(\FF_{16})$ are precisely $\{0, \z^4, \z^9, \z^{14}\}$. This set has $2^{(2,4)} = 2^2$ elements, as predicted, since $p^cx+ b = 2^2x + 3$ and $d = 4$.

\subsection*{Prime collusion}
We observed that the product of the Teichm\"{u}ller coefficients of $\alpha \in \ZZ[\z]$ over certain groupings of primes were often independent of the index $m$ and always restricted. For example, with $q = 2^4$ we had
\begin{align*}
    \p_1: \z^0 + \z^1 &= \,\,\z^4 + \z^8 p + \,\,\,\z^6 p^2 + \,\,\z^5 p^3 + \,\,\z^3 p^4 + 0\, p^5 + \z^8 p^6 + \z^{10} p^7 + \z^7 p^8 + \z^{10} p^9 + \ldots\\
    \p_2: \z^0 + \z^1 &= \z^{12} + \z^8 p + \z^{10} p^2 + \z^{11} p^3 + \z^{13} p^4 + 0\, p^5 + \z^8 p^6 + \,\,\z^{6} p^7 + \z^9 p^8 + \,\,\z^{6} p^9 + \ldots
\end{align*}
\noindent Then $\tau(m,\z^0 + \z^1,\p_1)\tau(m,\z^0+\z^1,\p_2) = \z^1$ appears to be true for each $m\geq 0$ when the product is not zero. To verify this, notice that $\z^0 + \z^1$ is invariant under the order two element $-x + 1 \in \aff(15)$. Proposition \ref{sub} tells us
\[
    \tau(m,\z^0+\z^1,\p_1)\tau(m,\z^0 + z^1,\p_2) = (ux + vb)\tau(m,\z^0+\z^1,\p_1),
\]
where $u = 1 + \sigma = 0$ and $v = 1 + (1 + \sigma) = 1$. Hence,
\[
    \tau(m,\z^0+\z^1,\p_1)\tau(m,\z^0 + z^1,\p_2) = \z^{vb} = \z^1.
\]
Since $u = 0$, the products are independent of $m$. Proposition \ref{2case} shows this is always the case for $\alpha = \z^{a_1} + \z^{a_2}$.

\begin{prop}
\label{2case}
If $\alpha = \z^{a_1} + \z^{a_2}$ with $a_1 \not\equiv a_2 \bmod q -1$, then $\alpha$ is invariant under $-x + a_1 + a_2$. Let $\overline{\p} := \p^\sigma$ when $\sigma = \m 1$. Then
\[
    \tau(m,\alpha,\p)\tau(m,\alpha,\overline{\p}) = \z^{a_1 + a_2}.
\]
\end{prop}

\begin{proof}
Verifying the invariance of $\alpha$ under $-x + a_1 + a_2$ is straightforward. We apply Proposition \ref{sub} with $\sigma = \m 1$ and $k = 2$. In this simple case we may evaluate the summations for $u$ and $v$ directly: $u = 1 + \sigma = 0$ and $v = 1 + (1 + \sigma) = 1$. We conclude that when $\tau(m,\alpha,\p)\neq 0$,
\[
    \tau(m,\alpha,\p)\tau(m,\alpha,\overline{\p}) = (ux + v(a_1+ a_2))\tau(m,\alpha,\p) = \z^{a_1 + a_2}.
\]
\end{proof}

With $q = 2^6$ we saw collusion in triplets of primes together with restricted coefficients.
\begin{align*}
    \p_1: \z^4 + \z^{37} + \z^{43}  &=\z^{56} + \z^{35} p + \z^{56} p^2 + \z^{14} p^3 + \z^{14} p^4 + \z^{35} p^5 + \z^{35} p^6 + \z^{56} p^7 + \z^{14} p^8 + \ldots\\
    \p_2: \z^4 + \z^{37} + \z^{43}  &=\z^{35} + \z^{14} p + \z^{35} p^2 + \z^{56} p^3 + \z^{56} p^4 + \z^{14} p^5 + \z^{14} p^6 + \z^{35} p^7 + \z^{56} p^8 + \ldots\\
    \p_3: \z^4 + \z^{37} + \z^{43}  &=\z^{14} + \z^{56} p + \z^{14} p^2 + \z^{35} p^3 + \z^{35} p^4 + \z^{56} p^5 + \z^{56} p^6 + \z^{14} p^7 + \z^{35} p^8 + \ldots\\
    &\\
    \p_4: \z^4 + \z^{37} + \z^{43}  &=\z^{35} + \z^{35} p + \z^{56} p^2 + \z^{14} p^3 + \z^{35} p^4 + \,\,\,0\, p^5\, + \,\,\,0\, p^6 + \z^{35} p^7 + \z^{14} p^8 + \ldots\\
    \p_5: \z^4 + \z^{37} + \z^{43}  &=\z^{56} + \z^{56} p + \z^{14} p^2 + \z^{35} p^3 + \z^{56} p^4 + \,\,\,0\, p^5\, + \,\,\,0\, p^6 + \z^{56} p^7 + \z^{35} p^8 + \ldots\\
    \p_6: \z^4 + \z^{37} + \z^{43}  &=\z^{14} + \z^{14} p + \z^{35} p^2 + \z^{56} p^3 + \z^{14} p^4 + \,\,\,0\, p^5\, + \,\,\,0\, p^6 + \z^{14} p^7 + \z^{56} p^8 + \ldots
\end{align*}
The element $\alpha = \z^4 + \z^{37} + \z^{43}$ is invariant under the order three subgroup generated by $16x + 42 \in \aff(H_2)$. The fixed points of $16x + 42$ are $\{0,\z^{14},\z^{35},\z^{56}\}$, hence the restricted coefficients in the expansions above by Proposition \ref{restrict}. The collusion in the triplets is caused by the invariance of $\alpha$ under the order three element $25 x + 0$. Using Proposition \ref{sub} we compute
\[
    u \equiv 1 + 25 + 25^2 \equiv 21 \bmod 63 = q - 1.
\]
The value of $v$ is irrelevant since $b = 0$. Hence
\[
    \tau(m,\alpha,\p)\tau(m,\alpha,\p^\sigma)\tau(m,\alpha,\p^{\sigma^2}) = \tau(m,\alpha,\p)^{21}.
\]
Since $21a \equiv 42 \bmod 63$ for $a = 14, 35, 56 \equiv 2 \bmod 3$ we conclude that
\[
    \tau(m,\alpha,\p)\tau(m,\alpha,\p^\sigma)\tau(m,\alpha,\p^{\sigma^2}) = \z^{42}
\]
whenever $\tau(m,\alpha,\p)\neq 0$.

\subsection*{Conclusion} Permutation conspiracies, restricted coefficient phenomena, and prime collusions are three readily apparent and seemingly unrelated patterns emerging in the Teichm\"{u}ller expansions of sums of roots of unity. All three are consequences of the linear action of the affine group $\aff(q-1)$ on $\ZZ[\zeta]$: permutation conspiracies occur between elements in the same orbit under an element of $\aff(H_p)$; restricted coefficients occur for elements fixed under some element of $\aff(H_p)$; and prime collusions occur for elements fixed under a general element of $\aff(q-1)$.

\subsection*{Acknowledgements} The author thanks Bob Lutz for helpful feedback on this manuscript.

\end{document}